\documentclass[oneside,english,12pt]{amsart}
\usepackage{amsthm}
\usepackage{hyperref}
\usepackage{txfonts}
\usepackage{amsfonts}
\usepackage{lmodern}
\usepackage{setspace}

\usepackage[T1]{fontenc}
\usepackage[latin9]{inputenc}
\usepackage{geometry}
\usepackage{amsthm}
\usepackage{amssymb}
\usepackage[all]{xy}
\usepackage{pinlabel}
\usepackage{dirtytalk}

\usepackage{mathtools}
\usepackage{amssymb,amsmath,amsthm,tikz}
\usepackage{pinlabel}
\usetikzlibrary{matrix,arrows,decorations.pathmorphing}

\title{Dynamics on the Morse Boundary}
\author{Qing Liu}

\makeatletter
\numberwithin{equation}{section}
\numberwithin{figure}{section}

\theoremstyle{plain}
\newtheorem{thm}{Theorem}[section]
\newtheorem{lem}[thm]{Lemma}
\newtheorem{cor}[thm]{Corollary}
\newtheorem{prop}[thm]{Proposition}

\theoremstyle{definition}
\newtheorem{defn}[thm]{Definition}
\newtheorem{eg}[thm]{Example}

\theoremstyle{remark}
\newtheorem{rmk}[thm]{Remark}

\newcommand{\bbR}{{\mathbb{R}}}
\newcommand{\bbZ}{{\mathbb{Z}}}

\newcommand{\bbN}{{\mathbb{N}}}

\usepackage{babel}

\tikzset{node distance=1.5cm, auto}

\begin{document}
\maketitle

\begin{abstract}
Let $X$ be a proper geodesic metric space and let $G$ be a group of isometries of $X$ which acts geometrically. Cordes constructed the Morse boundary of $X$ which generalizes the contracting boundary for CAT(0) spaces and the visual boundary for hyperbolic spaces. We characterize Morse elements in $G$ by their fixed points on the Morse boundary $\partial_MX$. The dynamics on the Morse boundary is very similar to that of a $\delta$-hyperbolic space. In particular, we show that the action of $G$ on $\partial_MX$ is minimal if $G$ is not virtually cyclic. We also get a uniform convergence result on the Morse boundary which gives us a weak north-south dynamics for a Morse isometry. This generalizes the work of Murray in the case of the contracting boundary of a CAT(0) space.

\end{abstract}

\section{Introduction}

In the study of hyperbolic spaces and hyperbolic groups, boundaries play a critical role. The visual boundary of a hyperbolic space, is  defined to be equivalence classes of geodesic rays, where two geodesic rays are equivalent if their Hausdorff distance is finite. This boundary is a quasi-isometry invariant. In particular, for a hyperbolic group, its boundary is well-defined.

In a CAT(0) space, we can define the visual boundary. But it is not always a quasi-isometry invariant. Charney and Sultan in \cite{CS2014} constructed a quasi-isometry invariant boundary for any complete CAT(0) space using
contracting rays. This boundary is called the contracting boundary.

A key property of geodesics in a hyperbolic space is the Morse property which guarantees that quasi-geodesics stay close to geodesics. Many non-hyperbolic spaces $X$ also contain some Morse geodesics and we can use these Morse geodesic rays to study hyperbolic-like behavior in $X$. Cordes in \cite{Cordes} constructed the Morse boundary for any proper geodesic space containing a Morse ray. The Morse boundary is homeomorphic  to the contracting boundary and the visual boundary in the cases of proper CAT(0) spaces and hyperbolic spaces, respectively. This boundary is a quasi-isometry invariant, so it is well-defined for a finitely generated group.

As we known, the visual boundary is very useful to study the structure of hyperbolic groups. The topology and dynamics on the boundary can give information about the geometry of the original space and its isometries. 

In the case of a CAT(0) space, Murray \cite{Murray} studied the topological dynamics of actions on the contracting boundary and he obtained many basic dynamical properties. Since the Morse boundary is a generalization of the contracting boundary, it is natural to ask if similar results hold. This is our goal in the current paper.

In this paper, we introduce and study the notion of a Morse isometry of a geodesic space $X$.  It is a generalization of a hyperbolic isometry in a hyperbolic space and a rank-one isometry in a CAT(0) space. Briefly a Morse isometry is an isometry $g$ which the orbit map is a quasi-isometric embedding of $\langle g \rangle$ into $X$ and its image is Morse in $X$. In particular, if we consider a finitely generated group and its action on the Cayley graph, a Morse isometry is also called a Morse element. We give a characterization of Morse elements in terms of their fixed points on the Morse boundary. The notation $Fix_{\partial_M{X}}(g)$ represents fixed points of $g$ on the Morse boundary. 

\newtheorem*{M in G}{\bf{Theorem \ref{M in G}}}
\begin{M in G}

 Let $G$ be a group acting geometrically on a proper geodesic space $X$. Let $g\in G$ be an infinite order element. Then the following are equivalent:
\begin{enumerate}
\item 
 $Fix_{\partial_MX}(g)$ is nonempty.

\item  
For some $x_0\in X$ (hence for any $x_0\in X$), there exists a Morse gauge $N_0$ depending only on $x_0$ and  $g$ such that the geodesic $[x_0, g^{k}(x_0)]$ is $N_0$-Morse for any $k\in \bbZ$.

\item
Let $\eta =[x_0, g(x_0)]$. The bi-infinite path $\eta_{-\infty}^{+\infty}=\bigcup_{k\in \bbZ} g^{k}(\eta)$ is a Morse quasi-geodesic. In particular, the element $g$ is Morse.

\item
 $Fix_{\partial_MX}(g)$ has two distinct points.
\end{enumerate}
\end{M in G}

We also show that dynamics of a Morse isometry $g$ are very simple. For example, considering the action of $g$ on $X\cup \partial_MX$, it contains two distinct fixed points on $\partial_MX$, we say $g^+$ and $g^-$. 
And for any point $p\in X \cup \partial_MX$ other than $\{g^-, g^+\}$, the sequence $\{g^{n}(p)\}$ converges to $g^+$ and $\{g^{-n}(p)\}$ converges to $g^{-}$. So $g^+$ and $g^-$ are attractor and repeller points. 

Actually, we prove a weak North-South dynamics for a Morse isometry.

\newtheorem*{weak N-S}{\bf{Corollary \ref{weak N-S}}}
\begin{weak N-S}

Let $X$ be a proper geodesic space and $x_0$ be a basepoint. 
Let $g$ be a Morse isometry of $X$.
Given any open neighborhood $U$ of $g^{+}$ in $\partial_MX_{x_0}$ and any compact set $K\subset \partial_MX_{x_0}$ with $g^-\notin K$. 
There exists an integer $n$ such that $g^{n}(K)\subset U$.
\end{weak N-S}

More generally, we can prove the following theorem. It says that the action on the Morse boundary behaves like a convergence group action.

\newtheorem*{con2 in M}{\bf{Theorem \ref{con2 in M}}}
\begin{con2 in M}

Let $X$ be a proper geodesic space and $x_0$ be a basepoint. 
Let $\{g_n\}$ be a sequence of isometries of $X$.
Assume that $g_n(x_0)\to p_1$ in $\partial_MX_{x_0}$ and $g_n^{-1}(x_0)\to p_2$ in $\partial_MX_{x_0}$. 
Given any open neighborhood $U$ of $p_1$ in $\partial_MX_{x_0}$ and any compact set $K\subset \partial_MX_{x_0}$ with $p_2\notin K$. 
There exists an integer $k$ such that for all $n>k$ we have $g_{n}(K)\subset U$ after passing to a subsequence.

\end{con2 in M}

For a finitely generated group G, we call it $non$-$elementary$ $Morse$ if it has nonempty Morse boundary and is not virtually cyclic. 
We also prove another topological dynamics theorem that the action of a non-elementary Morse group $G$ on its Morse boundary is minimal.

\newtheorem*{minimal}{\bf Theorem \ref{min}}
\begin{minimal}

If a finitely generated group $G$ is non-elementary Morse, then the action of $G$ on $\partial_MG$ is minimal, that is for any $p\in \partial_MG$ the orbit $Gp$ is dense in $\partial_MG$.

\end{minimal}

In our paper, the proofs of Theorem \ref{con2 in M} and Theorem \ref{min} are based on the following key lemma and its variants.

\newtheorem*{key 1}{\bf Lemma \ref{key 1}}
\begin{key 1}[Key Lemma]
Let $X$ be a proper geodesic space and $x_0$ be a basepoint. 
Let $\{g_n\}$ be a sequence of isometries of $X$.
Assume that $g_n(x_0)\to p_1$ $\in$  $\partial_MX_{x_0}$ and $g_n^{-1}(x_0)\to p_2$ $\in$ $\partial_MX_{x_0}$. Then for any point $q\in \partial_MX_{x_0}$ with $q\neq p_2$,  the sequence $g_n(q)$ converges to $p_1\in \partial_MX_{x_0}$.
\end{key 1}

Here is the outline of this paper. 
In section 2, we review the definitions of slim and thin triangle conditions. We give the construction and topology of the Morse boundary which is the main object in the paper.
In section 3. we list and prove some basic properties of Morse triangles and Morse geodesics which will be used repeatedly in the later sections. 
In section 4, we introduce the notion of a Morse isometry of a proper geodesic space. Using its fixed points in the Morse boundary, we give a characterization of the Morse element.
In section 5, in order to study dynamics, we need some preparations on the topology of the Morse boundary. 
In section 6, after proving the key lemma, we finish the proofs of all the above topological dynamics on the Morse boundary.

{\bf ACKNOWLEDGEMENTS.}
I would like to thank my advisor Ruth Charney for her enthusiasm and for encouraging me to write this paper.

\section{Preliminaries}

Let $(X, d)$ be a geodesic metric space. The notation $[x, y]$ represents a geodesic between two points $x, y\in X$. 
The metric space $X$ is called $proper$ if any closed ball in $X$ is compact.

\subsection{Slim Triangles and Thin Triangles}
In the later sections we will see that triangles with two $N$-Morse sides are hyperbolic-like triangles. Here we give a quick review about the slim and thin triangles conditions. All of these can be found in \cite[Chapter III.H]{bridson} or \cite{ghys1990groupes}.

\begin{defn}[Slim triangles]
Let $\delta$ be some non-negative constant.
A geodesic triangle in $X$ is called $\delta$-$slim$ if each of its sides is contained in the $\delta$-neighborhood of the union of the other two sides.
\end{defn}

\begin{defn}[Thin triangles]
Let $\delta$ be some non-negative constant. A triangle $\triangle(x, y, z)$ with sides $[x, y]$, $[y, z]$ and $[x, z]$ is called $\delta$-$thin$ if for any points $p\in [x, y]$ and $q\in [x, z]$ that satisfy $d(x, p)=d(x, q)\le \frac{1}{2}(d(x, y)+d(x, z)-d(y, z))$ we have $d(p, q)\le \delta$.
\end{defn}
 In a hyperbolic space, all triangles are uniformly slim and uniformly thin.

\subsection{Morse geodesics and the Morse boundary}

\begin{defn}[Hausdorff distance]
The Hausdorff distance $d_{\mathcal{H}}(A_1, A_2)$ between two subsets $A_1$ and $A_2$ is defined by $$\inf\{r\mid A_{1}\subset \mathcal{N}_{r}(A_2), A_{2}\subset \mathcal{N}_{r}(A_1)\},$$  where $\mathcal{N}_{r}(A_{i})$ is the $r$-$neighborhood$ of $A_{i}$. 
\end{defn}

\begin{defn}[Quasi-geodesics]
A map $f: (X, d_{X})\rightarrow (Y,d_{Y})$ between metric spaces is a $(\lambda, \epsilon)$-$quasi$-$isometric$ $embedding$, where $\lambda\ge 1, \epsilon>0$, if for any $x_1, x_2\in X$
$$\lambda^{-1}d_{X}(x_1, x_2)-\epsilon\le d_{Y}(f(x_1),f(x_2))\le \lambda d_{X}(x_1,x_2)+\epsilon.$$
If $X$ is an interval of $\bbR$, then the map $f$ is called a $(\lambda, \epsilon)$-$quasi$-$geodesic$. For convenience, we use the image of $f$ to describe the quasi-geodesic.  
\end{defn}

\begin{defn}[Morse (quasi)-geodesics]
Let $N$ be a function $[1, \infty)\times [0, \infty)\rightarrow [0, \infty)$. We say a (quasi)-geodesic $\gamma$ in a metric space is $N$-$Morse$, if for any $(\lambda, \epsilon)$-quasi-geodesic $\alpha$ with endpoints on $\gamma$, we have $\alpha\subset \mathcal{N}_{N(\lambda, \epsilon)}{(\gamma)}$. The function $N(\lambda,\epsilon)$ is called a $Morse$ $gauge$.
\end{defn}

This definition has its roots in a paper of Morse \cite{morse1924fundamental}. The Morse lemma says that every quasi-geodesic in a hyperbolic space is Morse. Morse geodesics in a proper geodesic space are similar to geodesics in a hyperbolic space. We would like to study hyperbolic-like behavior in more general spaces. 
The following lemma was proved in \cite[Corollary 2.5]{Cordes}. It says that the Morse geodesics are "hyperbolic directions" of the space.

\begin{lem}[Equivalent Geodesics, \cite{Cordes}]\label{E G}
Let $X$ be a geodesic metric space. Let $\alpha$ and $\beta$ be geodesic rays based at a point.
Suppose $\alpha$ is $N$-Morse and $d_{\mathcal{H}}(\alpha, \beta)$ is finite. Then there exists a Morse gauge $N_{1}$ and a constant $C_N$ such that $\beta$ is $N_{1}$-Morse and $d(\alpha(t), \beta(t))< C_N$ for all $t$, where $N_{1}$ and $C_N$ depend only on $N$.

\end{lem}

Now we give a review about the construction and topology of the Morse boundary. The reader can find more details and complete proofs in \cite{Cordes}.
\begin{defn}[The Morse boundary]
Let $X$ be a proper geodesic metric space and choose a basepoint $x_0\in X$. We say two Morse geodesic rays are $equivalent$ if their Hausdorff distance is finite. As a set, the $Morse$ $boundary$ $\partial_MX_{x_0}$ of $X$, is defined to be equivalence classes of all Morse geodesic rays with the basepoint $x_0$. In order to define the topology of the Morse boundary, choose a Morse gauge $N$. Define $$\partial_M^{N}{X_{x_0}}=\{[\alpha]\mid \mbox{there exists an } N\mbox{-Morse geodesic ray } \beta\in [\alpha] \mbox{ with basepoint } x_0 \}$$ with the compact-open topology.

Another equivalent way to define this topology on $\partial_M^NX_{x_0}$ is using a system of neighborhoods at a point in $\partial_M^NX_{x_0}$.
Let $\alpha$ be an $N$-Morse geodesic ray with $\alpha(0)=x_0$. 
For each positive integer $n$, let $V^{N}_n(\alpha)$ be the sets of $[\beta]\in \partial_M^NX_{x_0}$ with basepoint $x_0$ and $d(\alpha(t), \beta(t)) < C_N$ for all t< n, where $C_N$ is the constant from Lemma \ref{E G}. 
Then $\{V^N_n(\alpha)\mid n\in \bbN\}$ is a fundamental system of neighborhoods of $[\alpha]$ in $\partial_M^NX_{x_0}$.

Let $\mathcal{M}$ be the set of all Morse gauges. We put a natural partial ordering on $\mathcal{M}$ and define 
$$\partial_MX_{x_0}=\underset{\mathcal{M}}{\varinjlim} \partial_M^NX_{x_0}$$ with the induced direct limit topology, i.e., a set $U$ is open in $\partial_MX_{x_0}$ if and only if $U\cap\partial_M^NX_{x_0}$ is open for every Morse gauge $N$. 
\end{defn}

This topology has some nice properties. For a fixed Morse gauge $N$, the space $\partial_M^NX_{x_0}$ is compact and behaves like the boundary of a hyperbolic space.
With the direct limit topology, the Morse boundary is basepoint independent. So sometimes we will denote the Morse boundary as $\partial_MX$ without the basepoint.
We say a Morse boundary

\begin{thm}[Main Theorem in \cite{Cordes}]\label{MT of C}
Given a proper geodesic space $X$, with the direct limit topology, the Morse boundary $\partial_MX=\underset{\mathcal{M}}{\varinjlim} \partial_M^NX_{x_0}$ is a quasi-isometry invariant and is a visibility space, that is any two distinct points in $\partial_MX$ can be joined by a Morse bi-infinite geodesic. If $X$ is hyperbolic, the Morse boundary is homeomorphic to the visual boundary. If $X$ is $CAT(0)$, the Morse boundary is  homeomorphic to the contracting boundary.
\end{thm}

\begin{rmk}
Given a geodesic space $X$, Cordes and Hume in  \cite{cordes2017stability} constructed the $metric$ $Morse$ $boundary$ of $X$. It is a collection of boundaries and in the case of a proper geodesic space the direct limit of these boundaries is exactly the Morse boundary $\partial_MX$.
\end{rmk}

The following two corollaries of the Arzelà-Ascoli theorem \cite[Theorem 47.1]{munkres2000topology} are useful to us.
\begin{cor}\label{aa1}
Let $X$ be a proper geodesic metric space and $x_0\in X$. Let $\alpha_n:[0, a_n]\rightarrow X$ be a sequence  of geodesics such that $\alpha_n(0)=x_0$ and $a_n\to \infty$. Then the sequence $(\alpha_n)$ has a subsequence that converges uniformly on compact sets to a geodesic $\alpha:[0, \infty)\to X$.
\end{cor}

\begin{cor}\label{aa2}
Let $X$ be a proper geodesic metric space. Let $\alpha_n:[a_n, b_n]\to X$ be a sequence of geodesics such that $a_n\to -\infty, b_n\to\infty$, and the set $\{\alpha_n(0)\}$ has bounded diameter. Then the sequence $(\alpha_n)$ has a subsequence that converges uniformly on compact sets to a geodesic $\alpha:(-\infty,\infty)\to X$.
\end{cor}

\section{Basic properties of Morse geodesics}

In this section we will review and prove some basic properties about Morse geodesics which will be used in the later sections.

%
%

We now prove that all segments of a Morse quasi-geodesic are uniformly Morse.
\begin{lem}\label{seg M}
Let $X$ be a geodesic metric space.
Let $\alpha :I\to X$ be an $N$-Morse $(\lambda, \epsilon)$-quasi-geodesic where $I$ is an interval of $\bbR$. Then for  any interval $I'\subset I$, the $(\lambda, \epsilon)$-quasi-geodesic $\alpha'=\alpha \mid_{I'}$ is $N'$-Morse where $N'$ depends only on $\lambda,\epsilon$ and $N$.

\end{lem}

\begin{proof}

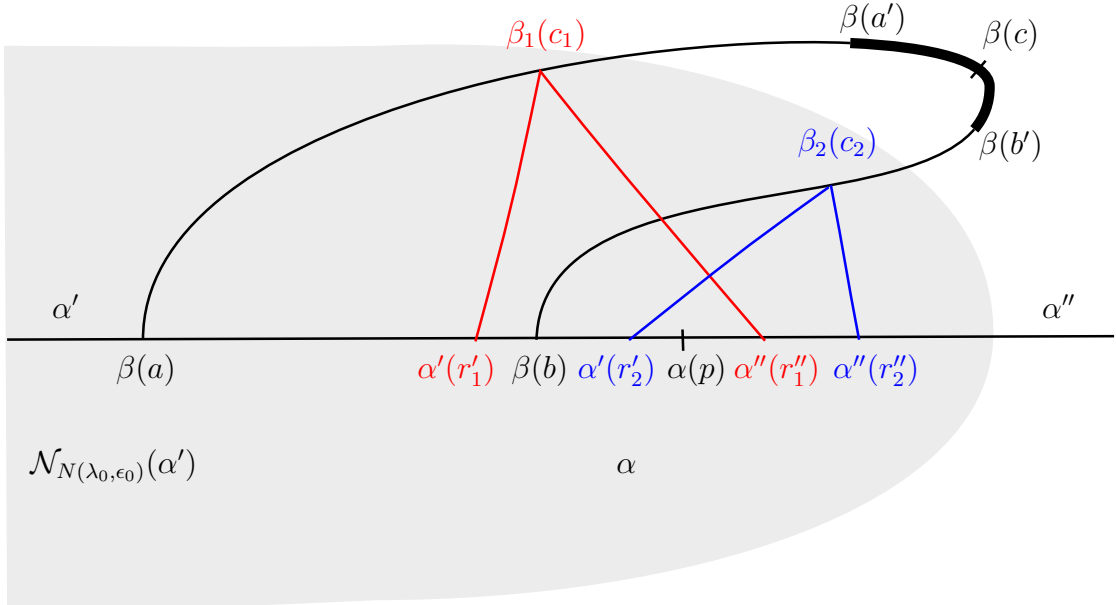
\begin{figure}[!ht]
\labellist

\pinlabel $\alpha$ at 200 50
\pinlabel $\alpha'$ at 25 100
\pinlabel $\alpha''$ at 335 100

\pinlabel $\beta(a)$ at 50 80
\pinlabel $\beta(b)$ at 173 80
\pinlabel $\beta(c)$ at 320 185
\pinlabel $\beta(a')$ at 277 190
\pinlabel $\beta(b')$ at 320 150
\pinlabel $\alpha(p)$ at 222 80

\pinlabel $\mathcal{N}_{N(\lambda_0, \epsilon_0)}(\alpha')$ at 40 50

{\color {red}
\pinlabel $\beta_1(c_1)$ at 175 185
\pinlabel $\alpha'(r_1')$ at 147 80
\pinlabel $\alpha''(r_1'')$ at 247 80
}
{\color{blue}
\pinlabel $\beta_2(c_2)$ at 266 151
\pinlabel $\alpha'(r_2')$ at 197 80
\pinlabel $\alpha''(r_2'')$ at 278 80
}

\endlabellist
\includegraphics[width=6in]{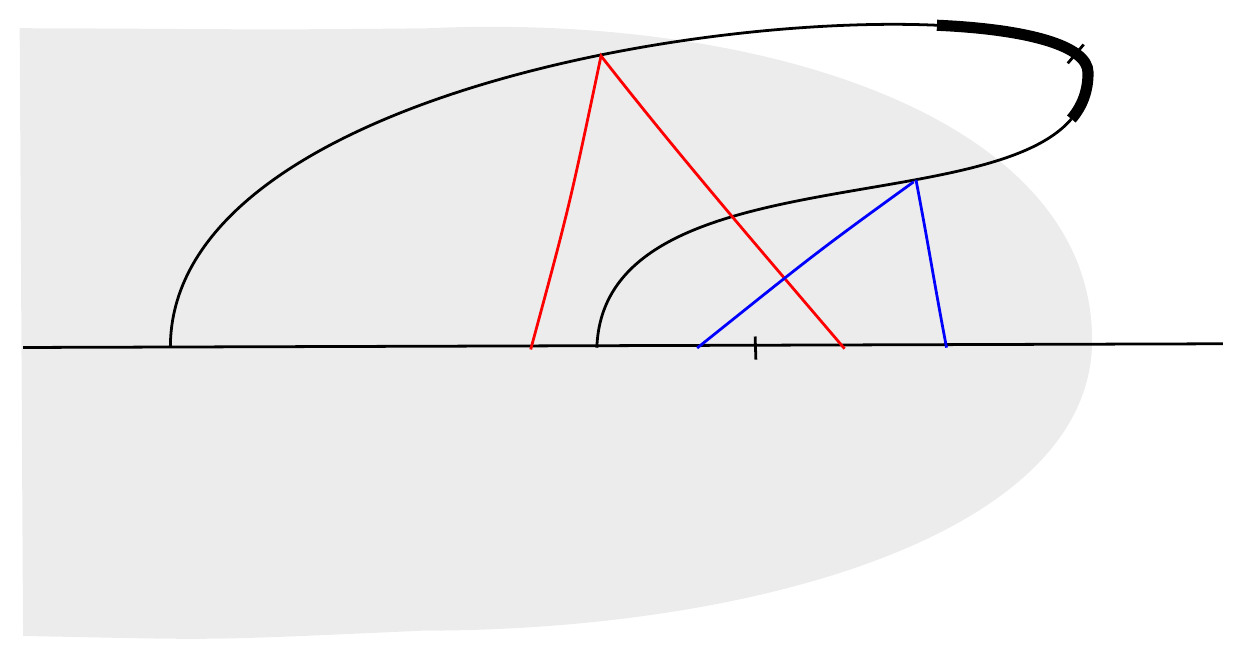}
\caption{Picture in Lemma \ref{seg M}. } 
\label{py}
\end{figure}

It suffices to prove this in the special case that $I$ and $I'$ have one endpoint (finite point or infinite point) in common, we can apply the special case twice for the general case.

Let $\beta: [a, b] \to X$ be a $(\lambda_0, \epsilon_0)$-quasi-geodesic with endpoints on $\alpha'$.
By Lemma 1.11 in \cite[Chapter III.H]{bridson}, we may assume without loss of generality that $\beta$ is tame. Denote $I''=I\setminus I'$, $\alpha''=\alpha\mid_{I''}$, $I'\cap I''=\{p\}$. 
By definition, $\beta \subset \mathcal{N}_{N(\lambda_0, \epsilon_0)}(\alpha)$. 
If $\beta \subset \mathcal{N}_{N(\lambda_0, \epsilon_0)}(\alpha')$, we are done.
If not, see Figure \ref{py} and consider some segment $[a', b'] \subset [a, b]$ such that $\beta([a', b'])$ is disjoint from $\mathcal{N}_{N(\lambda_0, \epsilon_0)}(\alpha')$. 
Choose $c\in (a', b')$ and let $\beta_1=\beta([a, c])$ and $\beta_2=\beta([c, b])$. 
For any $i=1, 2$, $\beta_i$ lies within $N(\lambda_0, \epsilon_0)$ of some point on $\alpha(I)$, so by continuity, there exist points $c_1\in [a, a']$ and $c_2\in[b', b]$ such that $\beta_i(c_i)$ lies within $N(\lambda_0, \epsilon_0)$ of two points $\alpha'(r'_i), \alpha''(r''_i)$, with $r'_i\in I', r''_i\in I''$. 
Since $\alpha$ is a $(\lambda, \epsilon)$-quasi-geodesic, we have $$\frac{1}{\lambda}|r'_i- r''_i|-\epsilon \le d(\alpha'(r'_i), \alpha''(r''_i))\le d(\alpha'(r'_i), \beta_i(c_i))+d(\beta_i(c_i), \alpha''(r''_i))\le 2N(\lambda_0, \epsilon_0).$$ Thus $|r'_i- r''_i| \le 2\lambda N(\lambda_0, \epsilon_0)+\lambda \epsilon$ for $i=1,2$. Note that $I'$ intersects $I''$ at exactly one point $p$, it follows that $|r'_1- r'_2| \le 2\lambda N(\lambda_0, \epsilon_0)+\lambda \epsilon$. 
Since $\alpha'$ is also a $(\lambda, \epsilon)$-quasi-geodesic$, d(\alpha'(r'_1), \alpha'(r'_2))\le \lambda(|r'_1- r'_2|)+\epsilon$.
By triangle inequality, 
$$d(\beta_1(c_1), \beta_2(c_2))\le d(\beta_1(c_1), \alpha'(r'_1))+d(\beta_2(c_2), \alpha'(r'_2))+d(\alpha'(r'_1), \alpha'(r'_2))$$
$$\le 2N(\lambda_0, \epsilon_0)+\lambda(2\lambda N(\lambda_0, \epsilon_0)+\lambda \epsilon)+\epsilon=(\lambda^2+1)(2N(\lambda_0, \epsilon_0)+\epsilon).$$

By the tameness condition, we have $length(\beta([c_1, c_2]))\le k_1(d(\beta_1(c_1), \beta_2(c_2)))+k_2$, where $k_1, k_2$ depend only on $\lambda, \epsilon$. 
Setting $N'_0(\lambda_0, \epsilon_0)=k_1(\lambda^{2}+1)(N(\lambda_0, \epsilon_0)+\frac{1}{2}\epsilon)+\frac{1}{2}k_2$, any point on $\beta([a', b'])$ lies within $N'(\lambda_0, \epsilon_0)$ of the two points $\beta(c_1)$ and $\beta(c_2)$.
We conclude that $\alpha'$ is $N'$-Morse where $N'(\lambda_0, \epsilon_0)=N'_0(\lambda_0, \epsilon_0)+N(\lambda_0, \epsilon_0)$.
\end{proof}

The next lemma gives a basic property of Morse geodesics. It is an easy exercise we leave to the reader.

\begin{lem}\label{close M}
Let $X$ be a geodesic metric space and $C$ be a constant. Suppose that the geodesic $[a, b]$ is $N$-Morse. Let $[a', b']$ be another geodesic. Suppose that $d(a, a')$ and $d(b, b')$ are bounded by $C$.
Then $[a', b']$ is $N'$-Morse and $d_{\mathcal H}([a,b], [a', b'])$ is bounded by $D$, where $N'$ and $D$ depend only on $N, C$.
\end{lem}

The following quite useful lemma is the combination of Lemma 2.3 and Lemma 2.4 in \cite{charney2018quasi}.
It states that a geodesic triangle with two $N$-Morse sides is slim and its third side is also Morse. It is important in showing the thinness of a Morse triangle. In later sectoions we will use it to show that certain sequences of geodesics are uniformly Morse.

\begin{lem}[\cite{charney2018quasi}]\label{slim 1}
Let X be a proper geodesic metric space. 
Let $\triangle(x, y, z)$ be a geodesic triangle with vertices $x, y, z\in X\cup\partial_MX$ and suppose that two sides of $\triangle(x, y, z)$  are $N$-Morse. Then the third side is $N_1$-Morse and $\triangle(x, y, z)$ is $\delta_N$-slim where $N_1$ and $\delta_N$ depend only on $N$.
\end{lem}

\begin{rmk}\label{rmk1}
By Lemma \ref{seg M} and Lemma \ref{slim 1}, we can see that for two points $p, q$ in the Morse boundary, all bi-infinite geodesics between p and q are uniformly Morse and have uniform Hausdorff distance. 
\end{rmk}

\begin{cor}\label{slim 2}
Let X be a proper geodesic metric space. Let $\triangle(x, y, z)$ be a geodesic triangle with vertices $x, y, z\in X\cup\partial_MX$ and suppose that two sides of $\triangle(x, y, z)$ are $N$-Morse. Let $\triangle(x', y', z')$ be any geodesic triangle with vertices $x', y', z'$ on the sides of $\triangle(x, y, z)$.
Then there exits a constant $\delta_N'$ depending only on $N$ so that $\triangle(x', y', z')$ is $\delta_N'$-slim.
\end{cor}

\begin{proof}
By Lemma \ref{slim 1}, the third side of $\triangle(x, y, z)$ is $N_1$-Morse. The points $x', y', z'$ divide the sides of $\triangle(x, y, z)$ into finitely many $N_2$-Morse geodesics by Lemma \ref{seg M} and each side of $\triangle(x', y', z')$ is $N_3$-Morse for some $N_3$ by Lemma \ref{slim 1}, where $N_1$, $N_2$ and $N_3$ depend only on $N$. Hence it is $\delta_{N_{3}}$-slim.
\end{proof}

With Corollary \ref{slim 2} in mind we can show that a triangle with two Morse sides is thin. Sometimes the thin triangle condition is easier to use than the slim triangle condition.
\begin{lem}\label{thin}
Let $X$ be a geodesic metric space. Let $\triangle(x, y, z)$ be a geodesic triangle with vertices $x, y, z\in X$ and suppose that two sides of $\triangle(x, y, z)$ are $N$-Morse. Then there exists a constant  $\delta''_N$ depending only on $N$ so that $\triangle(x, y, z)$ is $\delta''_N$-thin.
\end{lem}

\begin{proof}
By Lemma \ref{slim 1}, the triangle $\triangle(x, y, z)$ is $\delta_N$-slim.
For any vertices $x', y', z'$ on the sides of $\triangle(x, y, z)$, the triangle $\triangle(x', y', z')$ is $\delta'_N$-slim for some constant $\delta'_N$ depending on $N$ by Corollary \ref{slim 2}. The rest of the proof now follows that of Proposition 2.1 \cite{alonso1991notes}. The triangle $\triangle(x, y, z)$ is $\delta''_N$-thin by taking $\delta''_N=2\delta_N'+4\delta_N$.

\end{proof}

\section{Morse Isometries}

In this section we study fixed points of isometries on the Morse boundary. Moreover for a finitely generated group, we give a characterization of Morse elements by its fixed points on the Morse boundary. Given an isometry $g$ of $X$, let $Fix_{\partial_M{X}}(g)$ denote fixed points of $g$ on the Morse boundary.

\subsection{Morse isometries}
Before the next definition and lemma, we give some notations.

Let $X$ be a geodesic metric space and let $x_0\in X$ be a basepoint. For an isometry $g : X\to X$, denote $x_n=g^n(x_0)$, $\eta_{i}^{j}=[x_{i}, x_{i+1}]\cup...\cup[x_{j-1}, x_{j}]$, where $i, j\in \bbZ$. For convenience, we allow that $i={-\infty}$ or $j={\infty}$.

\begin{defn}[Morse Isometries]\label{Morse}
We say the action of $g$ on $X$ is $Morse$ or the isometry $g$ is $Morse$ if there exist a Morse gauge $N$ and a geodesic $[x_i, x_{i+1}]$ for each $i$ such that the bi-infinite concatenation of geodesics $\eta_{-\infty}^{\infty}$ is an $N$-Morse quasi-geodesic.
\end{defn}

Using the properties of Morse quasi-geodesics, it is easy to see that the notion does not depend on the choice of basepoint or choices of geodesics. In particular, we can choose a geodesic $[x_0, x_1]$ and take $[x_i, x_{i+1}]=g^{i}[x_0, x_1]$ for all $i$. Note that for a Morse isometry $g$, there is a $\mu > 0$ such that $d(x_n, x_0)\ge n\mu$ for all $n\in\bbN$.

\begin{defn}[Morse Elements]
Let $G$ be a finitely generated group.
We say $g\in G$ is a Morse element if the action of $g$ on the Cayley graph of $G$ for some (hence every) finite generating set is Morse.
\end{defn}

\begin{prop}
If $X$ is a $\delta$-hyperbolic space then every hyperbolic isometry of $X$ is Morse.
If $X$ is a proper $CAT(0)$ space then every rank-one isometry of $X$ is Morse.
\end{prop}

\subsection{Characterization of Morse elements}

The following lemma states that an infinite cyclic group which acts properly on a proper geodesic metric space has at most two fixed points on the Morse boundary. 

\begin{lem}\label{<3}
Let $X$ be a proper geodesic metric space. Let $g : X\to X$ be an isometry with infinite order. 
Suppose that the cyclic group $\langle g\rangle$ acts properly on $X$. Then $g$ has at most two fixed points on $\partial_{M}X$.
\end{lem}

\begin{proof}
Suppose $g$ has three distinct fixed points $a, b, c\in \partial_MX$. Choose some Morse gauge $N$ such that all geodesics with both endpoints in the set $\{a, b, c\}$ are $N$-Morse.
Given a constant $K\ge 0$, consider the set $E_K(a, b, c)=\{x\in X$| $x$ lies within $K$ of all sides of some ideal triangle $\triangle(a, b, c)$\}.
By Lemma 2.5 in \cite{charney2018quasi}, for any $K\ge \delta_N$, $E_K(a, b, c)$ is non-empty and has bounded diameter $L$ depending only on $N$ and $K$.

Now let $K=\delta_N$ and choose $p\in E_{\delta_{N}}(a, b, c)$, where $\delta_N$ is the constant from the Lemma \ref{slim 1}.
Since $a, b, c$ are fixed by $g$, then $g^{n}(p)\in E_{\delta_{N}}(a, b, c)$ for any $n\in \bbZ$.
But $E_{\delta_{N}}(a, b, c)$ has bounded diameter and the action of $\langle g\rangle$ on $X$ is proper, so $g$ must be finite order. We get a contradiction.
\end{proof}

Given a Morse isometry $g\in Isom(X)$ where X is a proper geodesic metric space, note that $g$ has no fixed points in $X$. The next lemma tells us that $g$ has two distinct fixed points on $\partial_MX$.
\begin{lem}\label{char M}
Let $X$ be a proper geodesic metric space and $x_0\in X$ be a basepoint. 
Let $g : X\to X$ be a Morse isometry.
There exists a bi-infinite Morse geodesic $\gamma$ such that $g$ has two distinct fixed points $\gamma(\infty),\gamma(-\infty)\in\partial _MX$.

\end{lem}

\begin{proof}
Suppose that the bi-infinite piecewise geodesic $\eta_{-\infty}^{\infty}$ is  an $N$-Morse $(\lambda,\epsilon)$-quasi-geodesic.
 By Lemma 2.5 (3) in \cite{CS2014} we know there exists $C$ depending only on $N, \lambda$ and $\epsilon$ such that the geodesic $[x_i, x_j]$ has Hausdorff distance at most $C$ from the quasi-geodesic $\eta_{i}^{j}$. With Lemma \ref{seg M} in mind, the quasi-geodesic $\eta_{i}^{j}$ is $M_1$-Morse where $M_1$ depends only on $N, \lambda$ and $\epsilon$. By Lemma 2.5 (1) in \cite{CS2014} we get that any geodesic between $x_i$ and $x_j$ is $M'$-Morse where $M'$ depends only on $M_1$ and $C$. 

 Let $\gamma_n$ be a geodesic segment from $x_{-n}$ to $x_{n}$. Let $\gamma_n(0)$ be some point in $\gamma_n$ such that $d(x_0, \gamma_n)=d(x_0, \gamma_n(0))$. Since $\gamma_n$ is $M'$-Morse and $\eta_{-n}^{n}$ is a $(\lambda,\epsilon)$-quasi-geodesic, so $d(x_0, \gamma_n)\le M'(\lambda, \epsilon)$. This implies that the set $\{\gamma_n(0)\}$ is bounded. By Corollary \ref{aa2}, the sequence $\gamma_n$ has a subsequence that converges uniformly on compact sets to a geodesic $\gamma:(-\infty, \infty)\to X$.  It is $M'$-Morse since every $\gamma_n$ is $M'$-Morse.

From the above proof, we have $d_{\mathcal{H}}(\gamma_n, \eta_{-n}^{n})\le C$. 
Since a subsequence of $(\gamma_n)$ converges uniformly on compact sets to $\gamma$ and the sequence $\eta_{-n}^{n}$ converges uniformly on compact sets to $\eta_{-\infty}^{\infty}$, we get  $d_{\mathcal{H}}(\gamma, \eta_{-\infty}^{\infty})\le C$. Thus, for any $t\in \bbR$, there exists $p_t\in \eta_{-\infty}^{\infty}$ such that $d(\gamma(t), p_t)\le C$. 
Note that by the triangle inequality $d(p_t, g(p_t))\le d(x_0, g(x_0))$. This implies that $d(\gamma(t), g(\gamma(t)))\le 2C+d(x_0, g(x_0))$ for all $t\in \bbR$. Note that the action of $\langle g \rangle$ is proper. By Lemma \ref{<3}, the fixed points of $g$ are exactly $\gamma(\infty)$ and $\gamma(-\infty)$.			 
\end{proof}

From the above proof, we can see that for a Morse isometry $g$ there exists a Morse gauge $N$ such that the geodesic $[g^{i}(x_0), g^{j}(x_0)]$ is $N$-Morse for every $i, j\in \bbZ$. The following theorem says that the converse is true in a finitely generated group. Now let us give a characterization of Morse elements.

\begin{thm}\label{M in G}
 Let $G$ be a group acting geometrically on a proper geodesic metric space $X$. Let $g\in G$ be an infinite order element. Then the following are equivalent:
\begin{enumerate}
\item 
 $Fix_{\partial_MX}(g)$ is nonempty.

\item  
For some $x_0\in X$ (hence for any $x_0\in X$), there exists a Morse gauge $N_0$ depending only on $x_0$ and $g$ such that the geodesic $[x_0, g^{k}(x_0)]$ is $N_0$-Morse for any $k\in \bbZ$.

\item
Let $\eta =[x_0, g(x_0)]$. The bi-infinite path $\eta_{-\infty}^{+\infty}=\bigcup_{k\in \bbZ} g^{k}(\eta)$ is a Morse quasi-geodesic. In particular, the element $g$ is Morse.

\item
 $Fix_{\partial_MX}(g)$ has two distinct points.
\end{enumerate}
\end{thm}

\begin{proof}
$(1)\Rightarrow(2)$
Fix a basepoint $x_0\in X$.
Let $q\in \partial_MX_{x_0}^N$ be one fixed point of $g$ for some Morse gauge $N$. Since $g$ is an isometry of $X$, we have an ideal triangle $\triangle (x_0, g^{k}(x_0), q)$ with two $N$-Morse sides $[x_0, q]$ and $g^{k}[x_0, q]$. By Lemma \ref{slim 1}, the side $[x_0, g^{k}(x_0)]$ is $N_0$-Morse for all $k\in\bbZ$, where $N_0$ depends only on $x_0$ and the action of $g$.

$(2)\Rightarrow(3)$
Suppose that $X$ is the Cayley graph of $G$ for some finite generating set. It is enough to show that the bi-infinite path $\eta_{-\infty}^{+\infty}$ is a Morse quasi-geodesic. 
The proof that $\eta_{-\infty}^{+\infty}$ is a quasi-geodesic follows the proof of Proposition 3.2 in \cite{alonso1991notes}. They use the property of thin triangles in hyperbolic space. In our case,  all triangles $\triangle(x_0, g^i(x_0), g^j(x_0))$ are $\delta$-thin for some constant $\delta$ by Lemma \ref{thin}, where $\delta$ depends only on $N_{0}$.
To see it is Morse, note that $[g^{i}(x_0), g^{j}(x_0)]$ is $N_0$-Morse and the path $\eta_{-\infty}^{+\infty}$ is a quasi-geodesic.  It follows that the path $\eta_{-\infty}^{+\infty}$ is Morse by definition.

$(3)\Rightarrow(4)$
This is Lemma  \ref{char M}.

$(4)\Rightarrow(1)$ This is trivial.

\end{proof}

For a Morse isometry $g$, its fixed points in $\partial_MX$ are denoted by $\{g^+, g^-\}$. We call $g^{+}$ and $g^{-}$ $poles$ or $rational$ $points$ of the Morse isometry $g$.
All Morse geodesics connecting $g^-$ and $g^+$ are called $axes$ of $g$.
They have a uniform Morse gauge depending only on the two fixed points of $g$ in $\partial_MX$ by Remark \ref{rmk1}. 
We say $g$ is $N$-$Morse$ if all its axes are $N$-Morse.
A finitely generated group $G$ is called $non$-$elementary$ $Morse$ if $\partial_MG$ is nonempty and $G$ is not virtually cyclic.

\section{The topology of the Morse boundary}
Before studying the dynamics of the Morse boundary, let us study its topology.

\begin {defn}\label{conv 1}
Let $X$ be a proper geodesic space and $x_0$ be a basepoint.
We say a sequence of points $p_n\in X$ $converges$ to a point $p\in \partial_MX$ with respect to $x_0$ (denoted by $p_n\to_{x_0} p$) if there exists some Morse gauge $N$ and some geodesic $[x_0, p_n]$ for any $n$ such that

(1) The geodesic $[x_0, p_n]$ is $N$-Morse.

(2) Every subsequence of $[x_0, p_n]$ has a subsequence that converges uniformly on compact set to a geodesic ray $\gamma$ with $\gamma(0)=x_0, \gamma(\infty)=p$.

\end{defn}

By the properties of Morse geodesics, the notion does not depend on the choice of geodesics $[x_0, p_n]$.
The next proposition says that the defination does not depend on the choice of basepoint. So we will use the notation $p_n\to p$. 

\begin{prop}
Let $X$ be a proper geodesic space and $x_0$ be a basepoint. Let $p_n\in X$ be a sequence of points and $p\in \partial_MX$ with $p_n\to_{x_0}p$. Then $p_n\to_{x}p$ for any $x\in X$.
\end{prop}

\begin{proof}
Suppose that the geodesic $[x_0, p_n]$ is $N$-Morse. Fix a point $x\in X$. By Lemma \ref{close M}, we know for any $n\in\bbN$, the geodesic $[x, p_n]$ is $N'$-Morse where $N'$ depends only on $N$ and $d(x_0, x)$. 
By the Arzelà-Ascoli theorem, every subsequence of $[x, p_n]$ has a subsequence $[x, p_{n_i}]$ that converges uniformly on compact set to some geodesic ray $\alpha(t)$.
Since $[x, p_n]$ is $N'$-Morse, the geodesic ray $\alpha$ is $N'$-Morse and $\alpha(0)=x$. 
Since $p_n\to_{x_0}p$, then $[x_0, p_{n_i}]$ has a subsequence that converges uniformly on compact set to some Morse geodesic ray $\gamma$ with $\gamma(0)=x_0, \gamma(\infty)=p$.
By Lemma \ref{close M}, $d_{\mathcal H}([x_0, p_n], [x, p_n])$ is bounded by a constant $D$ which depends only on $N, d(x, x_0)$.
So the same holds for $d_{\mathcal H}(\gamma, \alpha)$. We conclude that $\alpha(\infty)=p$.

\end{proof}

\label{conv 2}
Similarly as in Definition \ref{conv 1}, suppose that $p_n, p\in \partial_M{X}$. We say a sequence of points $p_n$ $converges$ to $p$ with respect to $x_0$ if there exist $N$-Morse geodesic rays $[x_0, p_n)$ for all $n$ and any subsequence of $[x_0, p_n)$ has a subsequence that converges uniformly on compact sets to some $N$-Morse geodesic ray $\gamma$ such that $\gamma(\infty)=p$. Since this notation also does not depend on choice of basepoint, we denote it by $p_n\to p$.

Note that the sequence $p_n\to p$ if and only if  $p_n$ converges to $p$ in the topology of $\partial_M^{N}X_{x_0}$ for some Morse gauge $N$ and some basepoint $x_0$. This implies $p_n$ converges to $p$ in the topology of $\partial_MX_{x_0}$. The following lemma says that the converse is true.

\begin{lem}\label{uni conv}
Suppose that $p_n, p\in \partial_MX_{x_0} $ and the sequence of points $p_n$ converges to $p$ in the topology of $\partial_MX_{x_0}$. Then there exists a Morse gauge $N$ such that $p_n, p\in \partial_M^{N}X_{x_0}$ and $p_n$ converges to $p$ in the topology of $\partial_M^{N}X_{x_0}$.
\end{lem}

\begin{proof}
We will follow the proof of Lemma 3.4 in \cite{Murray}. Let $\gamma_n=[x_0, p_n), \gamma=[x_0, p)$ be Morse geodesics and let $\gamma$ be $N_0$-Morse.
Suppose that $\gamma_n$ and $\gamma$ are not contained in $\partial_M^{N}X_{x_0}$ for any $N$.
Then we can choose a subsequence $\gamma_{n_i}$ of $N_i$-Morse geodesic rays , where $N_{i}\ge N_{i-1}+1$ for all $i$ and $\gamma_{n_i}$ is not $N_{i-1}$-Morse. 
Let $P=\{\gamma_{n_i}\}_{i\ge 1}$. Note that for all $N$, $P\cap \partial_M^NX_{x_o}$ is finite, so it is closed in $\partial_M^NX_{x_0}$. It follows that $P$ is closed in $\partial_MX_{x_0}$. So $p_n$ can not converge to $p$ in $\partial_MX_{x_0}$. We conclude that $p_n, p\in \partial_M^{N'}X_{x_0}$ for some $N'$.

Since $p_n\to p$ in $\partial_MX_{x_0}$, by the universal properties of direct limits, 
there exist geodesic rays $\alpha_n=[x_0, p_n)$ such that any subsequence of $\alpha_n$ contains a subsequence that converges uniformly on compact sets to some geodesic ray $\alpha=[x_0, p)$.
Note that all geodesics $\alpha_n, \alpha$ are $N''$-Morse for some $N''$, where $N''$ depends only on $N'$ by Lemma \ref{E G}. 
Since $p_n, p$ are all in $\partial_M^{N'}X_{x_0}$, this implies that $p_n\to p$ in $\partial_M^NX_{x_0}$, where $N=\max\{N', N''\}$.

\end{proof}

Now let us study the topology of $\partial_M^{N}X_{x_0}$.
Note that we have a natural inclusion map from $\partial_M^{N}X_{x_0}$ to $\partial_MX_{x_0}$. The next Lemma says that the induced subspace topology agrees with  the topology of $\partial_M^{N}X_{x_0}$.

\begin{lem}\label{sub top}
Let $B$ be a closed subset in $\partial_M^{N}X_{x_0}$. Then $B$ is a closed subset in $\partial_MX_{x_0}$. 
In particular, the topology of $\partial_M^{N}X_{x_0}$ is the same as the subspace topology of $\partial_MX_{x_0}$.
\end{lem}

\begin{proof}
It is enough to show that $B\cap\partial_M^{N'}X_{x_0}$ is closed in $\partial_M^{N'}X_{x_0}$ for all $N'$. Suppose it is nonempty. Choose $x_n\in B\cap\partial_M^{N'}X_{x_0}$ and suppose that $x_n\to x$ in $\partial_M^{N'}X_{x_0}$. There exist $N'$-Morse geodesic rays $\gamma_n$ and $\gamma$ with $\gamma_n(0)=x_0, \gamma_n(\infty)=x_n$ and $\gamma(0)=x_0, \gamma(\infty)=x$.
Also there exist $N$-Morse geodesic rays $\alpha_n$ with $\alpha_n(0)=0, \alpha_n(\infty)=x_n$. 
By the Arzelà-Ascoli theorem, any subsequence of $\alpha_n$ contains a subsequence that converges uniformly on compact sets to a geodesic ray $\alpha$ with $\alpha(0)=x_0$. 
The geodesic ray $\alpha$ is $N$-Morse since all of the $\alpha_n$ are $N$-Morse. Note that the Hausdorff distance between $\gamma_n$ and $\alpha_n$ is bounded by some constant $C$ depending only on $N$ and $N'$. 
So is the Hausdorff distance between $\alpha$ and $\gamma$. 
This means $\alpha(\infty)=x$ and $x_n\to x\in \partial_M^{N}X_{x_0}$. 
Since $B$ is a closed set in $\partial_M^{N}X_{x_0}$, this implies that $x\in B\cap\partial_M^{N'}X_{x_0}$. We conclude that $B\cap\partial_M^{N'}X_{x_0}$ is closed in $\partial_M^{N'}X_{x_0}$.

\end{proof}

The compact sets in $\partial _MX$ are studied in Cordes and Durham's paper \cite{cordes2016boundary}. 

\begin{lem}[Lemma 4.1 in \cite{cordes2016boundary}]\label{cpt in M}
Let $K$ be a compact subset in $\partial_MX_{x_0}$. Then there exists $N$ such that $K\subset \partial_M^{N}X_{x_0}$. 
\end{lem}

\begin{rmk}\label{cpt rmk}
From the two lemmas above, we know that if $K$ is a compact in $\partial_MX_{x_0}$, then $K$ is a compact set in $\partial_M^{N}X_{x_0}$ for some $N$.
\end{rmk}

Now we characterize when the Morse boundary of a group is compact. In the case of the contracting boundary of a $CAT(0)$ group, the reader can see Theorem 5.1 in \cite{Murray}.
\begin{thm}\label{cpt boundary}
Let $G$ be a group that is finitely generated and non-elementary Morse. Consider its Cayley graph $X$ with respect to a fixed finite generating set. The following are equivalent:

\begin{enumerate}
\item $G$ is hyperbolic.
\item $\partial_MX$ is compact.
\item The Morse gauges are uniform i.e. $\partial_MX_{x_0}\subset\partial_M^{N}X_{x_0}$ for some $N$.
\item For a point $p\in\partial_MX_{x_0}$, the orbits $Gp\subset\partial_M^{N}X_{x_0}$ for some $N$.

\item Geodesic segments in $X$ are uniformly Morse.
\item Geodesic rays in $X$ are uniformly Morse.
\end{enumerate}
\end{thm}

\begin{proof}
$(1)\Rightarrow(2)$ It follows Theorem \ref{MT of C}.\\
$(2)\Rightarrow(3)$ This follows Lemma \ref{cpt in M}.\\
$(3)\Rightarrow(4)$ It is trivial.\\
$(4)\Rightarrow(5)$ For any $g\in G$, the point $g(p)\in\partial_M^{N}X_{x_0}$. Consider the triangle with vertices $x_0, g(x_0) ,g(p)$. Note that $g[x_0, p]$ and $[x_0, g(p)]$ are  $N$-Morse. By Lemma \ref{slim 1}, the side $[x_0, g(x_0)]$ is $N_1$-Morse where $N_1$ depends only on $N$. We are done.\\
$(5)\Rightarrow(1)$ By Lemma \ref{slim 1}, all triangles are uniformly slim.\\
$(5)\Rightarrow(6)$ It follows from the definition of a Morse geodesic ray.\\
$(6)\Rightarrow(3)$ It is trivial.\\
\end{proof}

Now we can discuss the cardinality of the Morse boundary.  The following corollary answers Question 3.17 in \cite{cordes2016boundary}. It says that a group is  non-elementary Morse if and only if its Morse boundary contains infinitely many points. 

\begin{cor}\label{02infinity}
Let $G$ be a finitely generated group. The cardinality of its Morse boundary $\partial_MG$ is either $0$, $2$ or infinity.
\end{cor}

\begin{proof}

Suppose that $|\partial_MG|=n$ for some $n\ne0, \infty$.  It means that $\partial_MG$ is compact. By Theorem \ref{cpt boundary} $(2)$, the group $G$ is hyperbolic. So $n$ must be $2$.
\end{proof}

\section{Dynamics of the action on The Morse boundary}
In this section we will study topological dynamics of a group action on the Morse boundary. Most of following results have roots in the case of a hyperbolic space. Murray \cite{Murray} used dynamical methods to obtain many properties of the topological dynamics on the contracting boundary of a CAT(0) space. Since the Morse boundary is a generalization of the contracting boundary, it is interesting to ask which of them is true in the Morse boundary.

It is well-known that the action of a non-elementary hyperbolic group on its boundary is minimal. Murray \cite[Theorem 4.1] {Murray} showed this also holds in the case of the contracting boundary of a CAT(0) space. We generalize it to the Morse boundary.

\begin{thm}\label{min}
If a finitely generated group $G$ is non-elementary Morse, then the action of $G$ on $\partial_MG$ is minimal, that is for any $p\in \partial_MG$ the orbit $Gp$ is dense in $\partial_MG$.
\end{thm}

We postpone the proof of this theorem. From the above theorem, we can show that the set of rational points is either empty or dense in the Morse boundary.
\begin{cor}\label{dense 2}
If a finitely generated group $G$ is non-elementary Morse and contains a Morse element, then the set of rational points $Q(G):=\{ g^+|g\in G$ is a Morse element $\}$ is dense in $\partial_MG$.
\end{cor}

\begin{proof}
Consider the orbits $Gg^+$ where $g^+\in \partial_MG$ is a fixed point of the Morse element $g\in G$. The orbit $Gg^+ \subset Q(G)$ since $h(g^+)=(hgh^{-1})^{+}$ and $hgh^{-1}$ is Morse. So $Q(G)$ is dense in $\partial_MG$ since
$Gg^+$ is dense in $\partial_MG$.
\end{proof}

\begin{rmk}
In \cite{Fink}, she presents a torsion group which has nonempty Morse boundary. So we can not omit the assumption that $G$ contains a Morse element.
\end{rmk}

In hyperbolic groups, we have the classical North-South dynamics as follows. 

\begin{thm}
Let $G$ be a hyperbolic group acting geometrically on a proper geodesic space $X$. Let $g$ be an infinite order element. Then for all open sets $U$ and $V$ with $g^{+}\in U$ and $g^{-}\in V$, $g^n(V^{c})\in U$ for suffiently large $n\in \bbN$.
\end{thm}

In the case of a CAT(0) space, rank-one isometries act on the visual boundary with this classical North-South dynamics \cite{hamenstadt2009rank}. However, if we consider the action on the contracting boundary of a proper CAT(0) space, this may fail for a rank-one isometry. Murry \cite{Murray} gave the following example.

\begin{eg}[\cite{Murray}]
Let $Y$ be the space $T^{2}\vee S^{1}$. Consider the fundamental group $\pi_{1}(Y)=\bbZ*\bbZ^{2}=\langle a\rangle* \langle b,c \mid [b,c]\rangle$. Let $\tilde{Y}$ be the universal cover. It is a proper $CAT(0)$ space and $\pi_{1}(Y)$ acts geometrically on $\tilde{Y}$. Let $\alpha$ be an axis for the rank-one isometry $a$. 
Let $\beta_{i}$ be the geodesic with respect to the word $a^{-i}b^{i}aaaaa\ldots$. Note that the sequence of geodesics $\{\beta_{i}\}$ does not converge to $\alpha(-\infty)$ in the contracting boundary $\partial_c\tilde{Y}$ and $\{\beta_i\}$ is closed in $\partial_c\tilde{Y}$.  The set $V=(U(\alpha(-\infty), r, \epsilon)\cap\partial_C\tilde{Y}) \backslash \{\beta_i\}$ is an open neighborhood of $\alpha(-\infty)$. However, we have $a^{n}\beta_n$ is not in $(U(\alpha(-\infty), r', \epsilon')$ for all $n\in \bbN$ and $r'>\epsilon'$.
\end{eg}

Murray \cite{Murray} proved a weaker version of North-South dynamics on the contracting boundary of a CAT(0) space. Additionally, he showed that a CAT(0) group $G$ acts on its contracting boundary like a convergence group. We have analogous results for the Morse boundary. The proofs are given in later subsections.

\begin{thm}\label{con2 in M}
Let $X$ be a proper geodesic space and $x_0$ be a basepoint. 
Let $\{g_n\}$ be a sequence of isometries of $X$.
Assume that $g_n(x_0)\to p_1$ in $\partial_MX_{x_0}$ and $g_n^{-1}(x_0)\to p_2$ in $\partial_MX_{x_0}$. 
Given any neighborhood $U$ of $p_1$ in $\partial_MX_{x_0}$ and any compact set $K\subset \partial_MX_{x_0}$ with $p_2\notin K$. There exists an integer $k$ such that for all $n>k$ we have $g_{n}(K)\subset U$ after passing to a subsequence.
\end{thm}

\begin{rmk}
In Theorem \ref{con2 in M}, under the assumption that $g_n(x_0)\to p_1$, by the Arzelà-Ascoli theorem, it is not hard to find a subsequence $g_{n_i}$ such that $g_{n_i}^{-1}(x_0)\to p_2$ for some $p_2\in \partial_MX_{x_0}$. 
\end{rmk}

\begin{cor}[Weak North-South dynamics for Morse isometries]\label{weak N-S}
Let $X$ be a proper geodesic space and $x_0$ be a basepoint. 
Let $g$ be a Morse isometry of $X$.
Given any open neighborhood $U$ of $g^{+}$ in $\partial_MX_{x_0}$ and any compact set $K\subset \partial_MX_{x_0}$ with $g^-\notin K$. There exists an integer $n$ such that $g^{n}(K)\subset U$.
\end{cor}

\begin{figure}[!ht]
\labellist
\pinlabel $x_0$ at 320 32
\pinlabel $g_n(x_0)$ at 90 95
\pinlabel $g_n^{-1}(x_0)$ at 530 0
\pinlabel $p_1$ at 5 55
\pinlabel $p_2$ at 620 60
\pinlabel $q$ at 320 350 
\pinlabel $g_n(q)$ at 85 420

{\color {red}
\pinlabel $\gamma$ at  300 220
\pinlabel $g_n(\gamma)$ at  70  220
}
{\color {brown}
\pinlabel $\alpha_n$ at  170 70
\pinlabel $\beta_n$ at 400 10
}
{\color {teal}
\pinlabel $\gamma_n$ at  160 220
\pinlabel $g_n^{-1}(\gamma_n)$ at  380 220
}
\endlabellist
\includegraphics[width=4.5in]{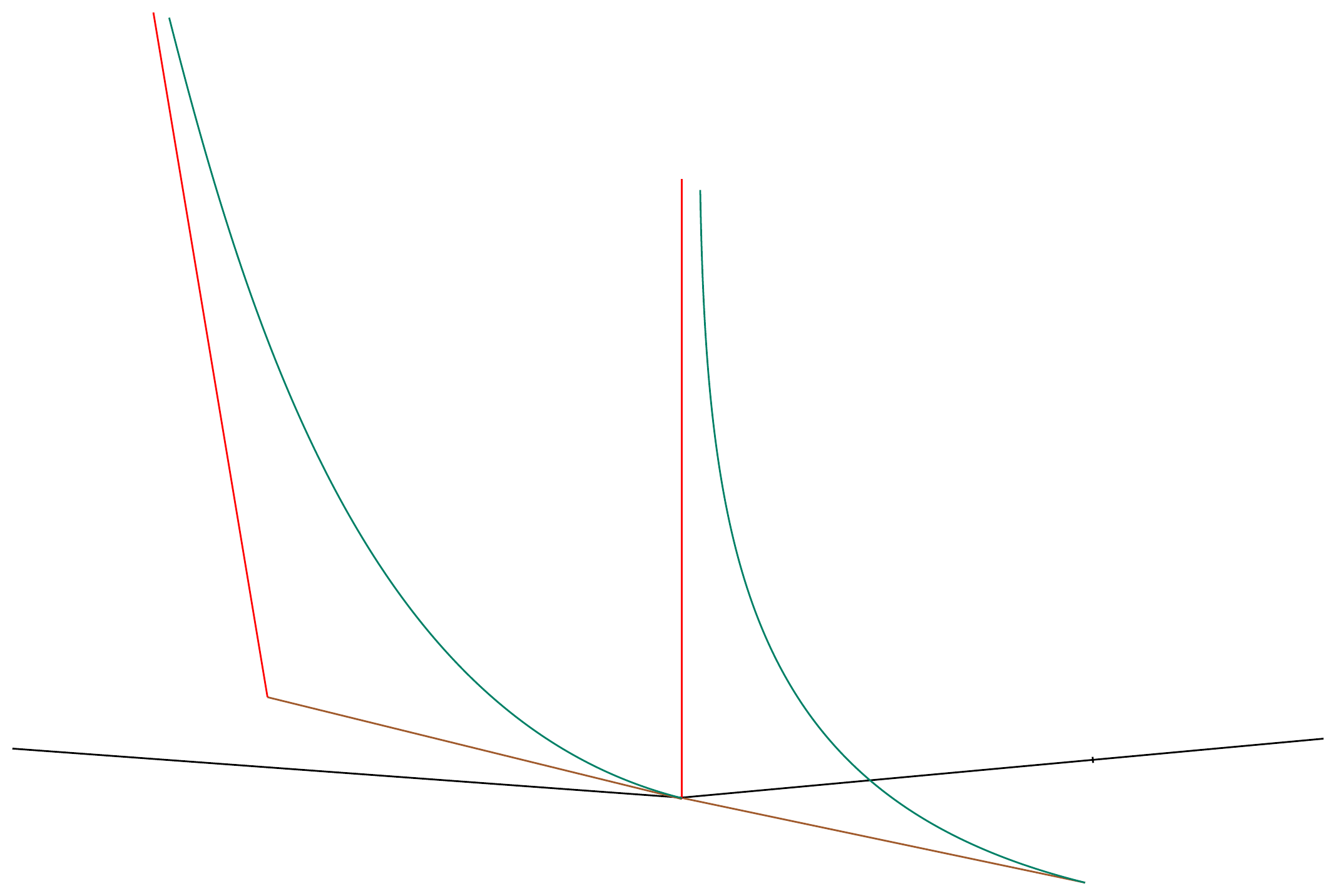}
\caption{Picture in Lemma \ref{key 1}. 
$\alpha_n=[x_0, g_n(x_0)]$, $\beta_n=g^{-1}_{n}(\alpha_n)$, 
$\gamma=[x_0, q)$, $\gamma_n=[x_0, g_n(q)]$.}
\label{scn}
\end{figure}
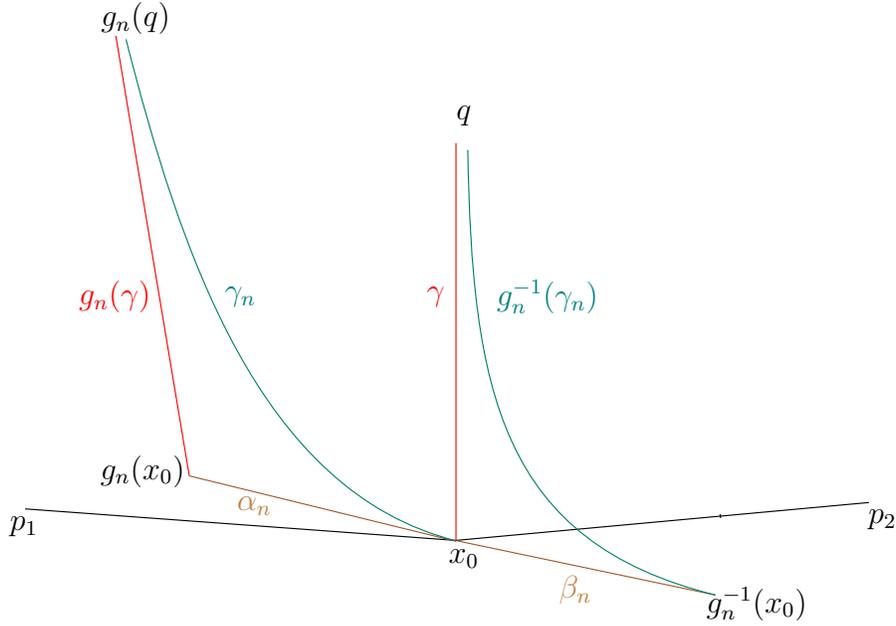

\subsection{Key Lemma} All of the above dynamical results on the Morse boundary rely on the following key lemma and its variants.
\begin{lem}[Key Lemma]\label{key 1}
Let $X$ be a proper geodesic space and $x_0$ be a basepoint. 
Let $\{g_n\}$ be a sequence of isometries of $X$.
Assume that $g_n(x_0)\to p_1$ $\in$  $\partial_MX_{x_0}$ and $g_n^{-1}(x_0)\to p_2$ $\in$ $\partial_MX_{x_0}$. Then for any point $q\in \partial_MX_{x_0}$ with $q\neq p_2$,  the sequence $g_n(q)$ converges to $p_1\in \partial_MX_{x_0}$.
\end{lem}

\begin{proof}
Let $\alpha_n=[x_0, g_n(x_0)], \beta_n=g^{-1}_{n}(\alpha_n)$ and $\gamma=[x_0, q)$. 
Assume that they are $N$-Morse for some Morse gauge $N$. 
Let $\gamma_n$ be a geodesic from $x_0$ to $g_n(q)$. By Lemma \ref{slim 1}, $\gamma_n$ is $N_1$-Morse, and this ideal triangle $\alpha_n\cup g_n(\gamma)  \cup \gamma_n$ is $\delta_N$-slim, where $N_1$ and $\delta_N$ depend only on $N$.
Let $N'=\max\{N, N_1\}$.
Since $\partial_M{X^{N'}_{x_0}}$ is compact, by passing to a subsequence, we may assume that $\gamma_{n}$ converges uniformly on compact sets to some geodesic ray $\eta$. It is enough to show that $\eta(\infty)=p_1$.

Since $g_n^{-1}(x_0)\to p_2$, by passing to a subsequence we have $\beta_{n}=g^{-1}_{n}(\alpha_n)$ that converges uniformly on compact sets to $\beta$, where $\beta$ is also an $N$-Morse geodesic from $x_0$ to $p_2$.
Note that $q\neq p_2$ if and only if $d_{\mathcal{H}}(\gamma, \beta)$ is unbounded, if and only if $\beta$ is not contained in $\mathcal{N}_K(\gamma)$ for any constant $K\ge 0$.
 
Now choose $K=2\delta_N+2$. There exists a point $b\in \beta$, such that $d(b, \gamma)> 2\delta_N+2$.
Now let us fix this point $b$. 
For any large $n$, we can find a point $b_n\in \beta_n$ such that $d(x_0, b)=d(x_0, b_n)$.

Since $\beta_n$ converges uniformly on compact sets to $\beta_N$, there exists $M_1>0$ such that $d(b, b_n)\le \delta_N+1$ for any $n>M_1$.
 Hence $d(b_n, \gamma)\ge \delta_N+1$, so $d(g_n(b_n), g_n(\gamma))\ge \delta_N+1$. Denote the point $a_n=g_n(b_n) \in \alpha_n$. We have $d(a_n, g_n(\gamma))>\delta_N$. Recall that the ideal triangle $\alpha_n\cup g_n(\gamma) \cup \gamma_n$ is $\delta_N$-slim. See Figure \ref{tri}. Thus there exists a point $c_n\in \gamma_n$ such that $d(c_n, a_n)\le \delta_N$ for any $n>M_1$. By Lemma \ref{seg M} and Lemma \ref{close M}, the geodesics $[x_0, a_n]$ and $[x_0, c_n]$ are $N_2$-Morse and the Hausdorff distance $d_{\mathcal{H}}([x_0, a_n], [x_0, c_n])\le C_0$, where $N_2$ and $C_0$ depend only on $N$.

Note that $d(x_0, g_n(x_0))-d(x_0, a_n)=d(g_n(x_0), a_n)=d(x_0, b_n)=d(x_0, b)$ is a fixed value for any $n$.
Thus $a_n\to p_1$.
By passing to a subsequence, $c_n\to\eta(\infty)$. Since $d_{\mathcal{H}}([x_0, a_n], [x_0, c_n])\le C_0$ we have  $\eta(\infty)=p_1$.
\end{proof}

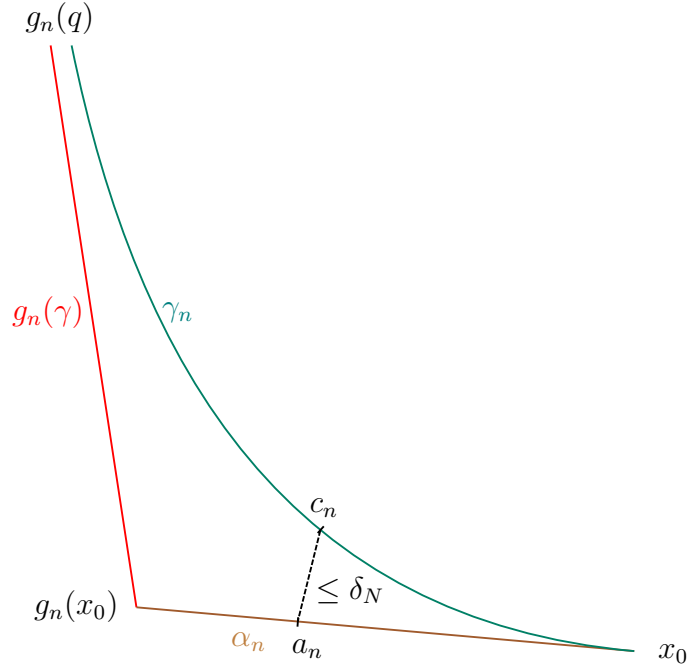
\begin{figure}[!ht]
\labellist
\pinlabel $x_0$ at 270 5
\pinlabel $g_n(x_0)$ at 17 25
\pinlabel $g_n(q)$ at  10 275
\pinlabel $a_n$ at 115 7
\pinlabel $c_n$ at 122 67
\pinlabel $\le \delta_N$ at 133 32
{\color {red}
\pinlabel $g_{n}(\gamma)$ at 5 150
}

{\color {teal}
\pinlabel $\gamma_{n}$ at 60 150
}
{\color {brown}
\pinlabel $\alpha_{n}$ at 90 10
}

\endlabellist
\includegraphics[width=3.2in]{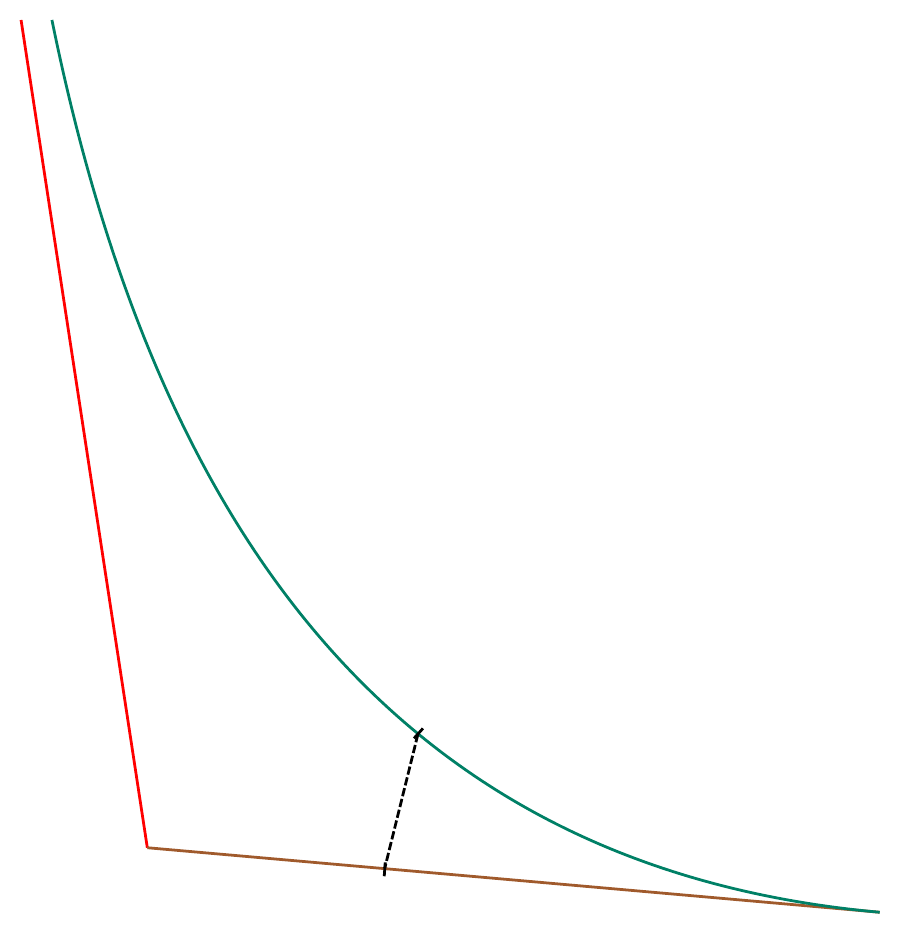}
\caption{The triangle $\triangle(x_0, g_n(x_0), g_n(q))$ is $\delta_N$-slim and $d(a_n, c_n)\le \delta_N$.} 
\label{tri}
\end{figure}

\subsection{Proof of Theorem \ref{min}}
The following propositions were first shown for the contracting boundary of a CAT(0) space by Murray in \cite{Murray}. Here we establish a different approach to the first one using Key Lemma \ref{key 1}.
\begin{prop}\label{dense 1}
Let $X$ be a proper geodesic space and let $G$ be a group acting cocompactly on $X$. 
Let $q\in \partial_MX$. Then $q$ is globally fixed by $G$ or its orbit $Gq$ is dense in $\partial_MX$.
\end{prop}

\begin{proof}
Fix a basepoint $x_0\in X$. Assume that $q$ is not globally fixed by $G$. So we have $h(q)\neq q$ for some $h\in G$.
Let $\alpha$ be any Morse geodesic ray with $\alpha(0)=x_0$. 
Since the action of $G$ is cocompact, there exists some constant $C$ and a sequence of elements $g_n\in G$ so that $d(\alpha(n), g_n(x_0))\le C$ for any $n\in \bbN$.
Let $\alpha_n=[x_0, g_n(x_0)]$. 
We can assume that $q, h(q), \alpha(\infty)\in \partial_M^{N_0}X_{x_0}$ for some $N_0$.
By Lemma \ref{close M} the geodesic $\alpha_n$ is $N_2$-Morse and $\alpha_n\subset \mathcal{N}_{C_1}(\alpha)$, where $N_2$ and $C_1$ depend only on $C$ and $N_0$.

By the Arzelà-Ascoli theorem, it is not hard to see that passing to a subsequence if necessary, $g_n(x_0)\to \alpha' (\infty)$ and $g^{-1}_n(x_0)\to \beta(\infty)$ for some geodesic rays $\alpha'$ and $\beta$. They are Morse since $\alpha_n$ is $N_2$-Morse.
Note that $\alpha(\infty)=\alpha'(\infty)$ since $\alpha_n\subset \mathcal{N}_{C_1}(\alpha)$.
 
If $\beta(\infty)\neq q$, we conclude that $g_n(q)$ converges to $\alpha(\infty)$ by Lemma \ref{key 1}. Otherwise we replace $q$ with $h(q)$.

\end{proof}

\begin{proof}[Proof of Theorem \ref{min}]
If $Gp$ is not dense in $\partial_MG$ for some $p\in \partial_MG$, then $Gp=p$ by Proposition \ref{dense 1}. By Theorem \ref{cpt boundary} $G$ is a hyperbolic group. Thus $G$ is virtually $\bbZ$ and $|\partial_MG|=2$. We get a contradiction.
\end{proof}

\subsection {Proof of Theorem \ref{con2 in M}}

We give another version of Lemma \ref{key 1} in the language of neighborhoods.
Recall that $\{V^{N}_n(\alpha)\}$ is a fundamental system of neighborhoods of the point $[\alpha]$ in $\partial_M^{N}X_{x_0}$, where $\alpha$ is an $N$-Morse geodesic ray with $\alpha(0)=x_0$.
\begin{lem}\label{key 2}

Let $X$ be a proper geodesic space and $x_0$ be a basepoint. 
Let $\{g_n\}$ be a sequence of isometries of $X$.
Assume that $g_n(x_0)\to p_1$ in $\partial_M^{N}X_{x_0}$ and $g_n^{-1}(x_0)\to p_2$ in $\partial_M^{N}X_{x_0}$. There exists a Morse gauge $N_1$ depending only on $N$ such that the following holds.

Given any point $q\in \partial_M^{N}X_{x_0}$ with $q\neq p_2$, there exists a neighborhood $V^{N}_{m'}(q)$ of $q$ in $\partial_M^{N}X_{x_0}$ such that for any neighborhood $V^{N_1}_{m}(p_1)$ of $p_1$ in $\partial_M^{N_1}X_{x_0}$, by passing a subsequence we have $g_{n}(V^{N}_{m'}(q))\subset V^{N_1}_m(p_1)$ for all $n>k$, where $k$ depends only on $N$, $m$ and the sequence $g_n$.
\end{lem}

\begin{proof}
Since $q\ne p_2$, we can choose sufficiently large $m'$ such that $p_2\notin V^N_{m'}(q)$.
For any $[\gamma'] \in V^N_{m'}(q)$. By Lemma \ref{E G}, the geodesic $\gamma'$ is $N'$-Morse where $N'$ depends only on $N$. Let $N_0=\max\{N, N'\}$. 
By Lemma \ref{slim 1}, the geodesic triangle $\triangle(x_0, g_n(x_0), g_n([\gamma']))$ is $\delta_{N_0}$-slim and the geodesic $[x_0, g_n([\gamma']))$ is $N'_0$-Morse, where $\delta_{N_0}$ and $N'_0$ depend only on $N$. Let $N_1=\max\{N'_0, N_0\}$.

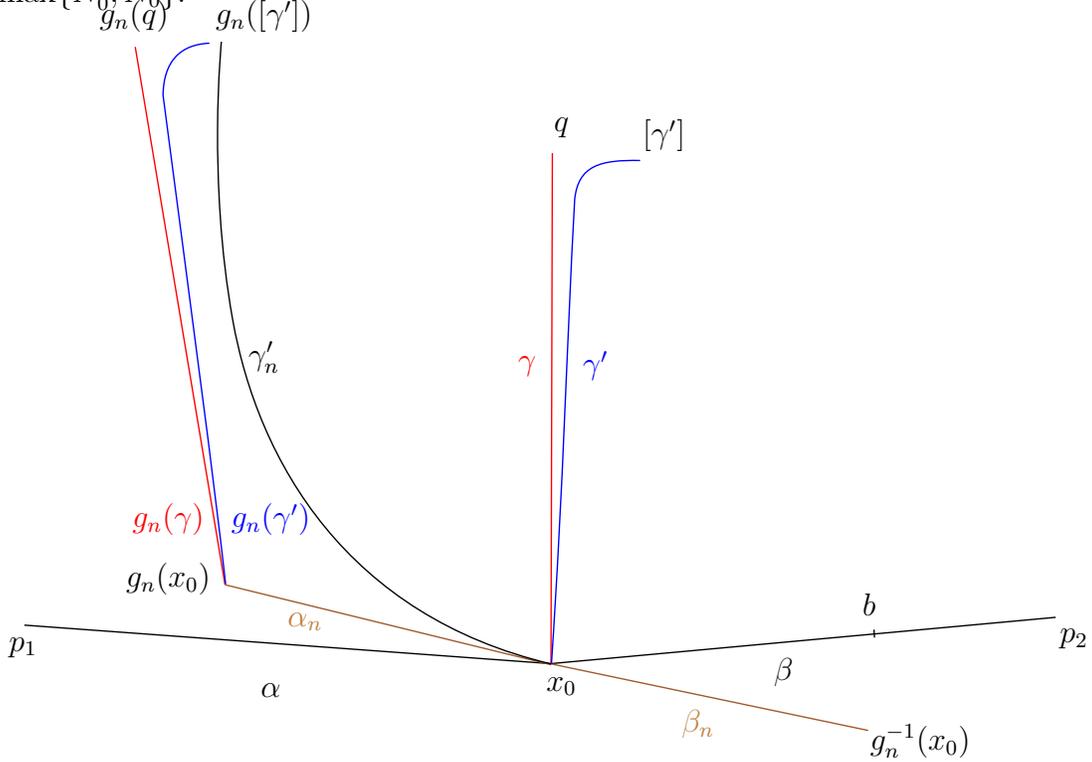
\begin{figure}[!ht]
\labellist
\pinlabel $\alpha$ at 150 30
\pinlabel $\beta$ at 450 40
\pinlabel $b$ at 500 80
\pinlabel $x_0$ at 320 32
\pinlabel $g_n(x_0)$ at 90 95
\pinlabel $g_n^{-1}(x_0)$ at 530 0
\pinlabel $p_1$ at 5 55
\pinlabel $p_2$ at 620 60
\pinlabel $q$ at 320 360 
\pinlabel $g_n(q)$ at 70 425
\pinlabel $[\gamma']$ at 380 355
\pinlabel $g_n([\gamma'])$ at 146 425
\pinlabel $\gamma'_n$ at 146 225

{\color {blue}
\pinlabel $\gamma'$ at  340 220
\pinlabel $g_n(\gamma')$ at  150  130
}
{\color {red}
\pinlabel $\gamma$ at  300 220
\pinlabel $g_n(\gamma)$ at  90  130
}
{\color {brown}
\pinlabel $\alpha_n$ at  170 70
\pinlabel $\beta_n$ at 400 10
}
\endlabellist
\includegraphics[width=5.5in]{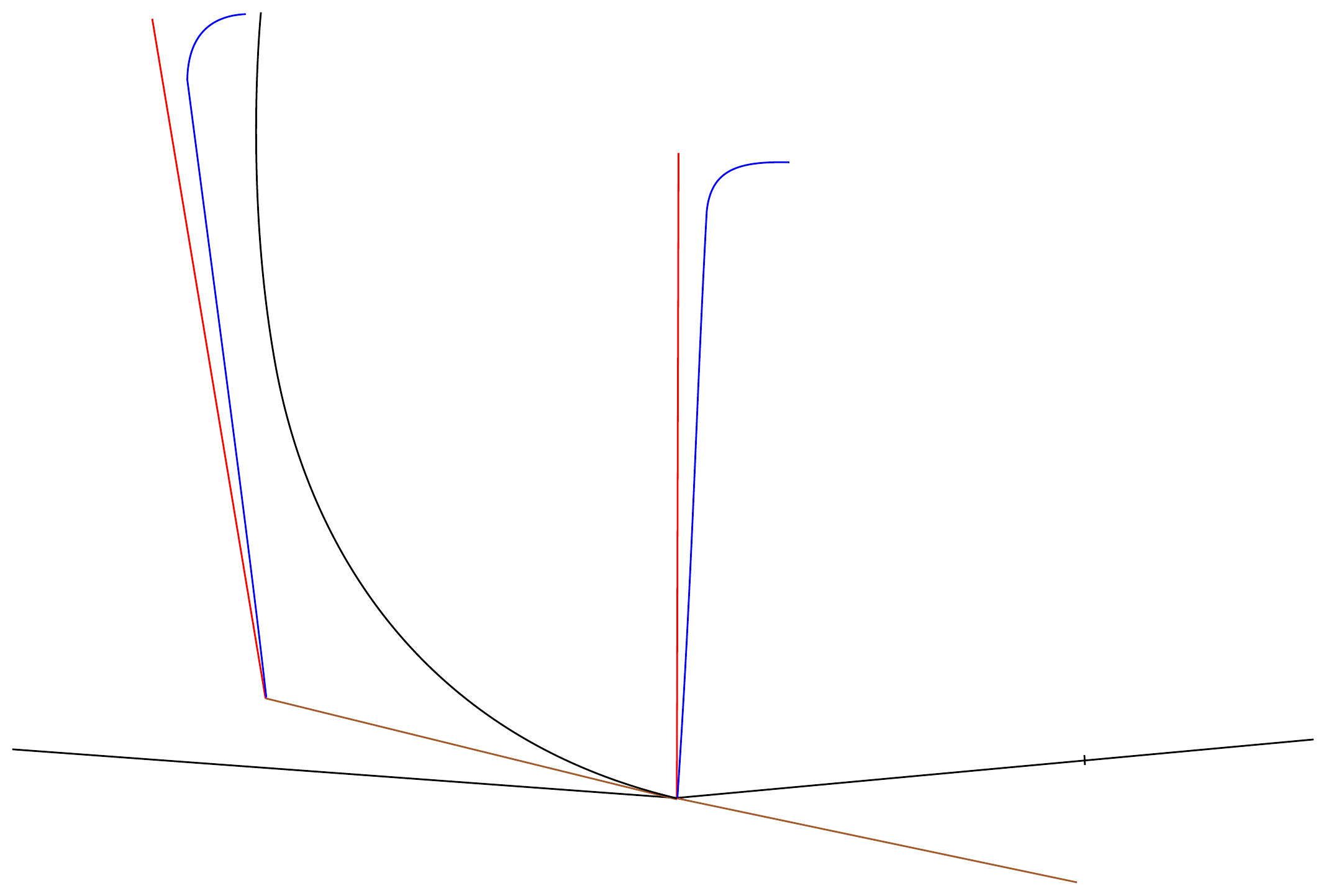}
\caption{Picture in Lemma \ref{key 2}. 
$\alpha_n=[x_0, g_n(x_0)]$, $\beta_n=g^{-1}_{n}(\alpha_n)$, 
$\gamma=[x_0, q)$, $[\gamma']\in V^{N}_{m'}(q)$. $b\in \beta$ and $d(b, \gamma')> 2\delta_{N_0}+2$ for all $[\gamma']\in V^N_{m'}(q)$}
\label{scn1}
\end{figure}

Now we use the same notations as in the proof of Lemma \ref{key 1}.  Let $\gamma=[x_0, q)$, $\beta=[x_0, p_2)$. 
Fix the point $b\in\beta$ such that $d(b, \gamma)> 2\delta_{N_0}+2+C_N$, where $C_N$ is the constant in Lemma \ref{E G}. 
Choose an integer $m'>d(b, x_0)+2\delta_{N_0}+2$.
Note that $d(\gamma'(t), \gamma(t))< C_N$ for all $t\in [0, m')$. Thus $d(b, \gamma')> 2\delta_{N_0}+2$ for all $[\gamma']\in V^N_{m'}(q)$. See Figure \ref{scn1}

Let $\gamma'_n=[x_0, g_n([\gamma']))$ and $\alpha_n=[x_0, g_n(x_0)]$.
Following the proof of Lemma \ref{key 1}, for $n> M_1$ we can find points $a_n\in \alpha_n$ and $b'_n\in \gamma'_n$ such that $d(a_n, b'_n)< \delta_{N_0}$ for all $[\gamma']\in V^N_{m'}(q)$. 

By the standard argument, $d(\gamma'_n(t), \alpha_n(t)) < C'_0$ for any $t\in [0, d(x_0, a_n)-\delta_{N_0}]$ and $n>M_1$, where $C'_0$ depends only on $N$. Note that $a_n\to p_1$. There exists $M_2$ such that $d(\gamma'_n(t), \alpha(t))< C'_0+1$ for all $n>M_2$ up to subsequence.
For any $m>0$, we can find $k$ large enough such that $d(x_0, a_n)-\delta_{N_0}>\max\{ m+2C'_0+2, 6C'_0+6\}$ for all $n>k$.
By Corollary 2.5 in \cite{Cordes}, we have $d(\gamma'_n(t), \alpha(t))< C_N$ for any $t< m$ and all $n>k$. It implies that $g_n([\gamma'])\in V^{N_1}_m(p_1)$ for all $[\gamma']\in V^{N}_{m'}(q)$ and $n>k$.
\end{proof}

\begin{proof}[Proof of Theorem \ref{con2 in M}]
By Lemma \ref{cpt in M}, there exists some Morse gauge $N$ such that $K\subset \partial_M^{N}X_{x_0}$ and by Remark \ref{cpt rmk} $K$ is compact in $\partial_M^{N}X_{x_0}$. By Lemma \ref{uni conv} we can choose large $N$ so that $g_n(x_0)\to p_1$ in $\partial_M^{N}X_{x_0}$ and $g^{-1}_n(x_0)\to p_2$ in $\partial_M^{N}X_{x_0}$. Note that $U\cap \partial_M^{N_1}X_{x_0}$ is an open neighborhood of $p_1$ in $\partial_M^{N}X_{x_0}$. 

Let $N_1$ be the Morse gauge in Lemma \ref{key 2}. Fix a neighborhood $V^{N_1}_m({p_1})\subset U\cap \partial_M^{N_1}X_{x_0}$ of $p_1$ in $\partial_M^{N_1}X_{x_0}$.
For any point $s\in K$, we can find some neighborhood $V^N_{n_s}(s)$ of $s$ in $\partial_M^{N}X_{x_0}\backslash p_2$. For each $s$ there exists an open neighborhood $U_s\subset V^N_{n_s}(s)$. Then $K$ is covered by $(U_s)_{s\in K}$. Since $K$ is compact, there exists a finite set ${s_i}, i=1,2,...,l$ and $K$ is covered by $\cup_{i=1}^{l}U_{s_i}$

By passing to a subsequence, for each $i$ we have $k_i$ such that $g_{n}(V^N_{n_{s_i}}(s_i))\subset V^{N_1}_m({p_1})$ for all $n\ge k_i$ by Lemma \ref{key 2}.
Taking $k=\max\{k_1, k_2,..., k_l\}$. We conclude that $g_n(K)\subset U\cap \partial_M^{N_1}X_{x_0}$ for all $n\ge k$.
\end{proof}

\subsection{Proof of Corollary \ref{weak N-S}}
In this subsection we will see that the dynamics of a Morse isometry $g$ is very simple. 
\begin{prop}\label{dy M1}
Let $g$ be a Morse isometry of a proper geodesic space $X$.
Then $g^{n}(x)\to g^{+}$ and $g^{-n}(x)\to g^{-}$ for any point $x\in X$. 
\end{prop}

\begin{proof}The proof here is similar to that of Lemma \ref{char M}.
For any $x\in X$. By the Arzelà-Ascoli theorem, any subsequence of $[x, g^{n}(x)]$ has a subsequence $[x, g^{n_{i}}(x)]$ that converges uniformly on compact set to a geodesic ray $\alpha$ with $\alpha(0)=x$. Since $[x, g^{n}(x)]$ is $N_x$-Morse for some $N_x$, so $\alpha$ is $N_x$-Morse. By Lemma \ref{close M}, the Hausdorff distance $d_{\mathcal{H}}([x, g^{n_{i}}(x)], g([x, g^{n_{i}}(x)])\le D$ for some $D$ depending only on $N_x$ and $d(x, g(x))$. It follows that $d_{\mathcal{H}}(\alpha, g(\alpha))\le D$. So $a(\infty)$ is a fixed point of $g$. It must be $g^+$. We conclude that $g^n(x)\to g^+$. 
\end{proof}

\begin{cor}\label{dy M2}
Let $g$ be a Morse isometry of a proper geodesic space $X$. Let $x_0$ be a basepoint of $X$.
Then the sequence $g^{n}(q)\to g^{+}$ as $n\to +\infty$ for any point $q\not=g^{-}$ in $\partial_MX_{x_0}$.  
\end{cor}

\begin{proof}
Let $g_i=g^{i}$. By Proposition \ref{dy M1},  the sequence of points $g^{n}(x_0)$( resp. $g^{-n}(x_0)$) converges to $g^+$( resp. $g^-$).
Since $q\not=g^{-}$, by Lemma \ref{key 1} we are done.

\end{proof}

\begin{proof}[Proof of Corollary \ref{weak N-S}]
The weak North-South dynamics theorem for Morse isometries follows from Theorem \ref{con2 in M} and Corollary \ref{dy M2}.
\end{proof}

\subsection{Schottky groups}
Given a proper geodesic space $X$, we say two Morse isometries $g_1$ and $g_2$ of $X$ are $independent$ if their fixed points on the $\partial_MX$ are disjoint. 

\begin{defn}\label{w-stable 1} Let $X$ be a geodesic space. Let $Y$ be a subset in $X$. We say $Y$ is $weakly$ $stable$ if there exists a Morse gauge $N$ such that any two points $y_1, y_2\in Y$ can be connected by an $N$-Morse geodesic in $X$. 
\end{defn}
\begin{defn}\label{w-stable 2}
Let $X$ be a proper geodesic space and let $x_0\in X$ be a basepoint. Let $G$ be a subgroup of $Isom(X)$.
We say $G$ is $weakly$ $stable$ with respect to the action if there exists a Morse gauge $N$ such that a geodesic $[x_0, g(x_0)]$ is $N$-Morse for all $g\in G$.
\end{defn}

The notion does not depend on choices of geodesics or the choice of basepoint.
For a finitely generated $G$, consider the natural action on its Cayley graph for some finite generating set. It is not hard to see that $G$ is weakly stable with respect to this action if and only if $G$ is a hyperbolic group.

\begin{prop}\label{rk 2}
Let $g_1, g_2$ be two independent Morse isometries of $X$.
Assume that the group $\langle g_1, g_2\rangle$ is weakly stable.
Then for sufficiently large $M\in \bbN$, elements $g_1^M$ and $g_2^M$ generate a nonabelian free subgroup of $\langle g_1, g_2\rangle$ and there exists a Morse gauge $N$ such that every nontrivial element in $\langle g_1^M, g_2^M \rangle$ is $N$-Morse.
\end{prop}

\begin{proof}
Let $x_0$ be a basepoint. Assume that the geodesic $[x_0, g(x_0)]$ is $N_0$-Morse for any $g\in \langle g_1, g_2\rangle$.
By Proposition \ref{dy M1}, by passing to a subsequence we may assume that for $i=1, 2$, the sequence $[x_0, g_i^{\pm n}(x_0)]$ converges uniformly on compact sets to an $N_0$-Morse geodesic ray $\gamma_{g_i}^{\pm}$ such that $\gamma_{g_i}^{\pm}(0)=x_0, \gamma_{g_i}^{\pm}( \infty)=g_{i}^{\pm}$.
Given some constants $C, m>0$, let $U_{i}^{\pm}=\{g(x_0)\in X \mid$ $g$ is a nontrivial element of $G$ and there exists a geodesic segment $\alpha_g$ between $x_0$ and $g(x_0)$ such that $d(\gamma_{g_i}^{\pm}(t), \alpha_g(t))\le C$ for $t< m \}$.
Now choose $C=1+C'_0$, where $C'_0$ is the constant in the proof of Lemma \ref{key 2}. 
There exists $m$ such that the sets $U_{i}^{\pm}$ are nonempty. 
We can choose very large $m$ such that sets $U_{i}^{\pm}$ are disjoint. Let $X_0=\{x_0\}\cup U_{1}^{+}\cup U_{1}^{-}\cup U_{2}^{+}\cup U_{2}^{-}$.

Since the group $\langle g_1, g_2\rangle$ is weakly stable, we can follow the similar proof in Lemma \ref{key 2} to find $M>0$ such that $g_{i}^{\pm M}(X_0-U_{i}^{\mp})\subset U_{i}^{\pm}$ for each $i$. Thus, by the standard ping-pong argument, $\langle g_1^{M}, g_2^{M} \rangle$
is free.

Now we show that every nontrivial element of $\langle g_1^{M}, g_2^{M} \rangle$ is $N$-Morse for some $N$.
It is enough to show it in the case $w$ is a cyclically reduced word in $\langle g_1^{M}, g_2^{M} \rangle$ since every word is conjugate to a cyclically reduced word and the conjugate of an $N'$-Morse element is still $N'$-Morse.

Given $\mu>0$, following the proof of Lemma \ref{key 1} we can choose large $M$ such that, $d(x_0, g_{i}^{\pm M}(w'(x_0)))>d(x_0, w'(x_0))+\mu$ for any $w'(x_0)\in X_0-U_{i}^{\mp}$. 
Thus, for any cyclically reduced word $w$ in $\langle g_1^{M}, g_2^{M} \rangle$, we have $d(x_0, w(x_0))>|w|\mu$.
It follows that $d(x_0, w^{n}(x_0))>n|w|\mu$ for $n>0$. 
We conclude that the bi-infinte piecewise geodesic $\eta_{-\infty}^{+\infty}=\bigcup_{k\in \bbZ} w^{k}(\eta)$ is a quasi-geodesic where $\eta =[x_0, w(x_0)]$. This path is also Morse since $[x_i, x_j]$ is $N_0$-Morse. 
So the element $w$ is a Morse isometry.  Following the proof in Lemma \ref{char M}, there exists an $N_0$-Morse bi-infinite geodesic $\gamma_{w}$ and $w$ fixes $\gamma_{w}(\infty)$ and $\gamma_{w}(-\infty)$. So $w$ is $N$-Morse where $N$ depends only on $N_0$.

\end{proof}

The next proposition is true by an argument to the proof of Proposition \ref{rk 2}.
\begin{prop}\label{rk k}
Let $g_1, \ldots ,g_k$ be $k$ pairwise independent Morse isometries of $X$.
Assume that the group $\langle g_1, \ldots ,g_k\rangle$ is weakly stable.
Then for sufficiently large $M\in \bbN$, elements $g_1^M,\ldots, g_k^M$ generate a nonabelian free subgroup of $\langle g_1, \ldots ,g_k\rangle$ and there exists a Morse gauge $N$ such that every nontrivial element in $\langle g_1^M,\ldots, g_k^M\rangle$ is $N$-Morse.
\end{prop}

\newcommand{\etalchar}[1]{$^{#1}$}


\end{document}